\documentclass[11pt]{amsart}

\usepackage{amssymb,amsthm}

\usepackage[all]{xy}

\usepackage{graphicx}
\usepackage{amscd}
\usepackage{amsmath}
\usepackage{amsxtra}
\usepackage{amsfonts}

\usepackage{xcolor}
\usepackage{mathrsfs}  
\usepackage{leftindex}
\usepackage[euler]{textgreek}

\newcommand{\ip}[2]{\left\langle #1\, , \,#2\right\rangle}
\newcommand{\cj}[1]{\overline{#1}}
\newcommand{\supp}{\operatorname*{supp}}
\newcommand{\bz}{\mathbb{Z}}

\newcommand{\br}{\mathbb{R}}

\newcommand{\bn}{\mathbb{N}}

\newtheorem{theorem}{Theorem}[section]
\newtheorem{proposition}[theorem]{Proposition}
\newtheorem{corollary}[theorem]{Corollary}
\newtheorem{lemma}[theorem]{Lemma}

\theoremstyle{definition}
\newtheorem{definition}[theorem]{Definition}

\newtheorem{remark}[theorem]{Remark}

\newtheorem{example}[theorem]{Example}
\theoremstyle{remark}

\theoremstyle{definition}

\renewcommand{\theclaim}{\textup{\theclaim}}

\newtheorem*{acknowledgements}{Acknowledgements}

\numberwithin{equation}{section}

\title[Frame Vector Group Representations and Amenability Properties]{Frame Vector Group Representations and Amenability Properties}

\author{Dorin Ervin Dutkay}
\address{[Dorin Ervin Dutkay] University of Central Florida\\
	Department of Mathematics\\
	4000 Central Florida Blvd.\\
	P.O. Box 161364\\
	Orlando, FL 32816-1364\\
U.S.A.\\} \email{Dorin.Dutkay@ucf.edu}
\author{Catalin Georgescu}
\address{[Catalin Georgescu] University of South Dakota\\
          Department of Mathematical Sciences\\
          414 E. Clark St. \\
          Vermillion, SD 57069\\
U.S.A. \\} \email{Catalin.Georgescu@usd.edu}
\author{Gabriel Picioroaga}
\address{[Gabriel Picioroaga] University of South Dakota\\
          Department of Mathematical Sciences\\
          414 E. Clark St. \\
          Vermillion, SD 57069\\
U.S.A. \\} \email{Gabriel.Picioroaga@usd.edu}
\keywords{Frames, group representations, amenable groups, Haagerup's property, property (T) }
\subjclass[2010]{20F65, 22D10, 42C15, 46C05}

\begin{document}

 \dedicatory{Dedicated to Professor David Larson, with friendship and appreciation.}

\begin{abstract}
We provide a new characterization of amenability for countable groups, based on frame representations admitting almost invariant vectors. By relaxing the frame inequalities, thereby weakening amenability, we obtain a large class of countable groups which we call {\it framenable}. We show that this class has some permanence properties, stands in contrast with property (T), and contains, for example, all free groups $\mathbb{F}_n$, $\textup{Aut}(\mathbb{F}_2)$ and $\textup{Aut}(\mathbb{F}_3)$, all (countable) lattices of $SL(2,\mathbb{R})$, the Baumslag–Solitar groups $BS_{p,q}$, the braid groups $B_n$, and Thompson’s group $F$. 
\end{abstract}

\maketitle

\tableofcontents

\section{Introduction}
Frame sequences originated in engineering as useful tools in signal processing and data compression. The theoretical study of frames in Hilbert spaces began with the seminal work of Duffin and Schaeffer \cite{DS}. The applications and practical aspects of frames have consistently been a focus of research. For example, properties such as robustness to erasures and straightforward implementation have been sought in both earlier and more recent constructions. To this end, the structured data underlying the frame is intermediary, not the main object of interest. Its role is to make the frame capable of processing signals (vectors in a Hilbert space) as effectively as possible. 

At the same time, a large body of theoretical work has been, and still is, dedicated to frames in abstract settings. For example, the subject has been connected to functional analysis and operator theory \cite{HL}. Group and $C^*$-algebra representations have been used extensively in Parseval frame and wavelet constructions, e.g., see \cite{DJ,DJ06,HL,BJK,DHJ,CDPW}. Frame constructions with respect to fractal measures began with \cite{JP}, and have generated sustained research, motivated in part by the (still) open question of whether the middle-third Cantor measure admits Fourier frames \cite{St}.

\par The present contribution belongs more to the area of classifying countable groups by their representation theory on Hilbert spaces. Our motivation and starting point are the following result of Han and Larson \cite{HL}: any Parseval frame representation of a countable discrete group $G$ is equivalent to a subrepresentation of the left-regular representation of the group. We build on this idea and show that in the presence of almost invariant vectors, together with weaker frame vectors than Parseval ones, such a group turns out to be amenable (see Theorem \ref{amen}). We note in passing that Theorem \ref{amen} vastly improves results obtained in \cite{GP}. 

While amenability is considered a {\it soft} property, strongly opposed to the {\it rigid} property (T), our characterization reveals a certain tension: a representation that weakly contains the trivial representation (i.e., admits almost invariant vectors) yet is separated from it, as reflected in the totality and Bessel inequalities satisfied by weak frame vectors. Nevertheless, amenability is not straightforward to check in general  by finding almost invariant vectors in the left regular representation. In the case of abelian (countable) groups, our characterization highlights amenability as a manifestation of the Fourier transform admitting almost invariant vectors together with a special orthonormal basis in the $L^2$-space of the Pontryagin dual. More generally, in Theorem \ref{thai1} we characterize the existence of almost invariant vectors for representations associated to regular Borel measures on the dual. We extend these observations to the case of amenable and ICC (infinite conjugacy classes) groups: any frame representation must admit almost invariant vectors (see Theorem \ref{thai3}).

\par As explained at the start of Section \ref{e}, there are many ways (and motivations) to weaken amenability. Our approach (see Section \ref{wfr}) is based on relaxing the frame vector requirements while retaining almost invariant vectors in a representation. The generalization, under which we call a group {\it framenable}, consists simply of choosing a weak frame of the form $\{\pi_g(w)\}_{g\in S}$ where $S\subset G$ is a countable set. In particular, if $S=G$ then $G$ is amenable. This new class of countable groups is quite large, as can be seen from the results and examples in Sections \ref{wfr} and \ref{ex}. In fact, the only examples we know so far that do not belong to this class are the countable property (T) groups. However, framenability does not share the same stability properties as amenability or its weaker variants; for instance, it does not pass to subgroups. We prove, however, that it passes up from subgroups of finite index to the parent group (see Theorem \ref{ind}). It is worth noting that the induced representation used in Theorem \ref{ind} yields an interesting dilation result in frame theory (see Theorem \ref{framext}; the precise definitions are included in that section): if a countable subgroup (finite index not required) has a frame representation, then one can dilate it to a frame representation of the parent group. 

Combining Theorem \ref{ind} with the observation that groups admitting surjective morphisms onto framenable groups are themselves framenable allows us to identify many further examples (see Section \ref{ex}). We end the last section with a list of questions and problems that we hope will clarify the position of the framenable class of groups within the class of non-property (T) groups.

\section{Amenability and Frames}\label{e}
The main result of this section, Theorem \ref{amen}, connects amenability with frame representations. Then we take a closer look at abelian and amenable ICC (infinite conjugacy classes) group representations which admit almost invariant vectors, see Theorems \ref{thai1} and \ref{thai3}.  
Amenability was first introduced by von Neumann. The absence of amenability for the free group with two generators helped explain the Banach Tarski paradox. The property is important in classification results in the study of $C^*$ and von Neumann algebras. In group theory, amenable groups can be viewed as the natural next step in tractability beyond abelian groups. This connection, however, is far from immediate—it emerges either through the Markov–Kakutani fixed point theorem or via Følner sequences together with the classification of abelian groups. For more details on group representations and amenability see for example \cite{BO,BHV}. For the abelian case see, e.g., \cite{Folland}. 
For more details on frames we refer the reader to \cite{HL} and \cite{Chris}. 

\medskip

{\bf Note.}  In what follows, the term {\it countable} refers to a one-to-one and onto correspondence with the set of integers, thus a countable set is an {\it infinite} set. The neutral element of group $G$ is denoted by $e$ or $\text{id}_G$ unless otherwise specified. The left regular representation of group $G$ is denoted by $\lambda$, whereas its trivial representation is denoted by $1_G$. Hilbert spaces will always be separable, with complex-valued inner product $\langle, \rangle$, which we sometimes label if we wish to distinguish between different Hilbert spaces within the same context.   

\begin{definition} Let $H$ be a separable Hilbert space and $S$ a countable set. 
	\begin{enumerate}
\item A set $\{f_n \}_{n \in S}\subset H$ is a {\it frame} if there exist $A,B>0$ such that for all  $v \in H$ : 
\begin{equation} \label{frame}
A \|v\|^2 \leqslant \sum_{n \in S} |\langle v, f_n\rangle |^2 \leqslant B \| v\|^2.
\end{equation}
\item A frame $\{f_n \}_{n \in S}$ is a {\it Parseval frame} if $A=B=1$ in \eqref{frame}.
\item  $\{f_n \}_{n \in S}$ is a {\it Bessel sequence} if the second inequality in \eqref{frame} holds for all $v\in H$.
\item $\{f_n \}_{n \in S}$ is a {\it weak frame} if it is Bessel and total (i.e., its linear span is dense in $H$).
\end{enumerate}
\end{definition}

Item (iv) above occurred multiple times in the literature albeit not with this particular name (see \cite{Web,Rief} and the references therein). We coined it ``weak frame'' because by the uniform boundedness principle, a sequence $\{f_n \}_{n \in S}$ satisfies (iv) if and only if  
$$ 0< \sum_{n\in S} |\langle f_n, v  \rangle| ^2 <\infty\text{ for all } v\neq 0 $$
Clearly  any frame is a weak frame, but the converse is not true in general. 
\begin{example}
Let $\{e_n\}_{n\in\mathbb{N}}$ be an orthonormal basis in $H$ and $\{a_n\}_{n\in\mathbb{N}}$ be any sequence of nonzero real numbers converging to zero. Define the sequence $\{f_n\}_n$ in $H$, $f_n=a_ne_n$ for every $n$. For arbitrary $v\in H$ we have
\[
\sum_{n} |\langle v, f_n \rangle |^2 = \sum_{n} |a_n|^2 |\langle v,e_n \rangle |^2 , 
\]
If the sum on the left was zero for some $v \neq 0$ in $H$, then $\langle v,e_n\rangle = 0$ for every $n$ which implies $v=0$ as $\{e_n\}_n$ is an ONB.  The Bessel property follows from the fact that $\{a_n\}$ is bounded. So $\{f_n\}_n$ is a weak frame.  It cannot be a frame since if there was $A>0$ such that 
\[
A \|v\|^2 \leqslant \sum_{n} |\langle v,f_n \rangle |^2, \text{ for all } v\in H
\] 
then taking $v=e_n$ we will get $A \leqslant |a_n|^2$ for any $n$, hence $A=0$.

The following example serves the same purpose:  for all $n\in\bn$ take $f_n:=e_n+e_{n+1}$. Then $\{f_n\mid n\geq 1\}$ is a Bessel sequence with dense linear span, hence a weak frame. However, $\{f_n\mid n\geq 1\}$ cannot be a frame, see  \cite[Example 5.4.6]{Chris}.
\end{example}
Recall from \cite{HL} that a Parseval frame is the image of an orthonormal basis under a projection, and that a frame is precisely the image of a Parseval frame under a bounded invertible operator. We do not need Riesz bases in this paper, we just mention (see \cite{HL}) that a Riesz basis is the image of an orthonormal basis under a bounded invertible operator. These bounded operators are constructed using the {\it frame transform}, also called the {\it analysis operator}. 
Given a countable set $S$ and a frame $\{f_n\}_{n\in S}$ one defines the {\it analysis operator} 
$\theta: H\to l^2(S)$ 
\[
\theta(v)(m):=\langle v , f_m \rangle \quad\text{ for all }v\in H, m\in S
\]
The adjoint $\theta^*$ is called {\it the synthesis operator} and $\theta^*\theta$ is called {\it the frame operator}. Inequalities \eqref{frame} show that $\theta$ is well-defined, bounded, and one-to-one. 
Denoting the standard orthonormal basis of the (infinite dimensional) space $l^2(S)$  by $\{\delta_n\}_{n\in S}$ we have
\[
\theta(v) =\sum_{n \in S} \langle v , f_n \rangle \delta_n \quad\text{ for all }v \in H.
\]

The frame inequalities \eqref{frame} also imply that the range of the frame transform is closed:
\[
\| \theta(x)-\theta(y)\|^2=\sum_{n \in S}| \langle x-y,f_n \rangle |^2 \geqslant A \| x-y \|^2,
\]
so if $\{\theta(x_n)\}$ is a convergent sequence in $\text{range}\,\theta$, we will have that $\{x_n\}_{n \in S}$ is a Cauchy sequence in $H$ and hence convergent to some $x$. Then $\theta(x_n)$ converges to $\theta(x)$, in the range of $\theta$. Accordingly $\theta:H \to \text{range}\,\theta$ is continuous and invertible and with continuous inverse from the bounded inverse theorem. Notice that a weak frame is precisely the image of an orthonormal basis under a bounded operator whose adjoint is injective. This follows from the observation that, given a weak frame, $\theta$ and $\theta^*$ are well-defined, bounded,  $\theta^*(\delta_n)=f_n$, for all $n\in S$, and $\theta$ is injective.

\begin{definition}\label{fvr}
Let $G$ be a countable discrete group, and $\pi: G\to \mathcal U(H)$ a unitary representation. 
A vector $\varphi\in H$ is :
\begin{enumerate}
\item  An {\it ONB vector} if $\{\pi_g(\varphi))\}_{g\in G}$ is an orthonormal basis.
\item  A {\it Parseval frame vector} if $\{\pi_g(\varphi)\}_{g\in G}$ is a Parseval frame.
\item A {\it frame vector} if $\{\pi_g(\varphi)\}_{g\in G}$ is a frame.
\item A {\it weak frame vector} if $\{\pi_g(\varphi)\}_{g\in G}$ is a weak frame. 
\end{enumerate}
In the case when a vector as in either (i)-(iv) exists, we call $\pi$ a {\it ONB/Parseval frame/frame/weak frame representation} of the group $G$. 
\end{definition}

 In the case of countable, discrete groups, amenability displays many equivalent definitions. We will add to the list more in Theorem \ref{amen}, which shows a somewhat surprising connection with frame vector representations.

\begin{definition}\label{camn} A countable discrete group $G$ is {\it amenable} if there exists a left-invariant state $\varphi : l^{\infty}(G)\to\mathbb{C}$ (i.e., $\varphi$ is a positive, linear functional, invariant under left-multiplication, and $\varphi(e)=1$ ). 
\end{definition}
\begin{definition} Let $G$ be a countable discrete group and $\rho : G\to\mathcal U(H_1)$, $\pi : G\to\mathcal U(H_2)$  two unitary representations of $G$. 
	We write $\rho<\pi$ if $\rho$ is equivalent to a subrepresentation of $\pi$. 
	
	 The representation $\rho$ is {\it weakly contained} in $\pi$, denoted by $\rho\prec \pi$, if 
for any $v\in H_1$ and  any $F\subset G$ finite, and for all $\epsilon > 0$,  there exist $w_i\in H_2$, $i=1,\dots,n$, such that  
$$| \langle \rho_g v, v  \rangle - \sum_{i=1}^n\langle  \pi_g w_i,w_i\rangle | < \epsilon,\quad  \text{ for all }g\in F $$

\end{definition}

One can easily see that weak-containment is a transitive relation. We say that the representations $\pi$ and $\rho$ are {\it weakly equivalent} if $\pi\prec \rho$ and $\rho\prec\pi$. 

In the case of particular representations, weak-containment can be spelled out more clearly. Let $G$ be a countable discrete group and $\pi$ a unitary representation of $G$ on a Hilbert space $H$. Denote by $1_G$ the trivial representation of $G$ on the same $H$. The following characterizations can be found in many textbooks, see e.g., \cite{BHV} and \cite{HV} :

(i) $1_G\prec\pi$ iff 
 there exists a net $(v_i)_{i\in I}$ in $H$, $\|v_i\|=1$ for all $i$ and
\begin{equation}\label{ainv}
\quad \|\pi_g v_i-v_i\| \to 0\text{ for all } g\in G.
\end{equation}
We note that if $G$ is finitely generated then $1_G\prec \pi$ iff \eqref{ainv} holds for any generator $g\in G$. 

Recall that  for a countable discrete group $G$ the left regular representation is $$\lambda : G\to \mathcal U( l^2(G)),\quad\lambda_g(\varphi)(h)=\varphi(g^{-1}h),\mbox{ for all $g, h $ in $G$.}$$
The following characterization is sometimes used as definition:

 (ii) $G$ is amenable iff  $1_G\prec\lambda$.

\begin{definition} We say that a unitary group representation $\pi$ has {\it almost invariant vectors} if there is a net of unit  vectors in $H$ such that \eqref{ainv} holds. 
\end{definition}

We will need the following ingredient, used often in \cite{HL}. We adjusted it to our setting as a lemma, with proof. 
\begin{lemma}\label{commutation}
Let  $\pi:G \to \mathcal U(H)$ be a unitary representation of a countable discrete group $G$ having a weak frame vector $v$ and let $\theta$ be its associated operator. Then, for any $g \in G$, the following diagram is commutative:

\begin{center}
\begin{xy}
(240,20)*+{H}="a"; (260,20)*+{l^2(G)}="b";%
(240,0)*+{H}="c"; (260,0)*+{l^2(G)}="d";%
{\ar "a";"b"}?*!/_2mm/{\theta};
{\ar "b";"d"} ?*!/_2mm/{\, \lambda_g};
{\ar "a";"c"} ?*!/^2mm/{\pi_g\,};
{\ar "c";"d"}?*!/_2mm/{\theta};  
\end{xy}
\end{center}
\end{lemma}
\begin{proof}
If $\{\pi_h(v)\}_{h \in G}$ is a weak frame then $\theta$ is a well-defined, bounded operator.

\[
\theta(x)=\sum_{h \in G}\ip{x}{\pi_h(v)}\delta_h.
\]
We have:  

\begin{align*}
  \theta(\pi_g(x))&=\sum_{h \in G}\ip{\pi_g(x)}{\pi_h(v)}\delta_h= \sum_{h \in G}\ip{x}{\pi_{g^{-1}h}(v)}\delta_h=\\
  &=\sum_{g_1 \in G}\ip{x}{\pi_{g_1}(v)}\delta_{gg_1}=\lambda_g\left( \sum_{g_1 \in G}\ip{x}{\pi_{g_1}(v)}\delta_{g_1}\right)=\\
  &=\lambda_g(\theta(x)).
\end{align*}
\end{proof} 

\begin{theorem}\label{amen}Let $G$ be a countable discrete group. The following statements are equivalent:
	\begin{enumerate}
\item $G$ is amenable.
\item   There exists a representation $\pi$ of $G$ with ONB vector and $1_G\prec \pi $.
\item   There exists a representation $\pi$ of $G$ with Parseval frame vector and $1_G\prec \pi$.
\item   There exists a representation $\pi$ of $G$ with frame vector and $1_G\prec \pi$.
\item There exists a representation  $\pi:G \to \mathcal U(H)$ with weak frame vector and $1_G\prec\pi$.
\item  There exists a representation $\pi:G \to \mathcal U(H)$ with $1_G\prec\pi$ and  there exists $w\in H$ such that 
$\{\pi_g w\}_{g\in G}$ is total and 
$\sum_g |\langle v, \pi_g w \rangle|^2 <\infty$  for all $v$ in a dense subspace of $H$.
\end{enumerate}
\end{theorem}

Before proving the theorem above we will justify why items (i) through (iv) are all equivalent.  After this, we will close the circuit by showing (vi)$\Rightarrow$(i)  with the help of two lemmas.

 If $G$ is amenable then $1_G\prec \lambda$.  Also the vector $v=\delta_e$, where $e$ is the identity in $G$ is an ONB vector for $\lambda$. The implications (ii)$\Rightarrow$(iii)$\Rightarrow$(iv)$\Rightarrow$(v)$\Rightarrow$(vi) are all trivial.  We argue that (iv) implies the stronger containment $\pi < \lambda$, i.e., $\pi$ is a subrepresentation of the left regular representation. This appears in \cite{HL}, but we describe here the main ingredients. 
 
 \begin{lemma}\label{lemsub}
 	If a representation $\pi$ of $G$ has a frame vector then it has a Parseval frame vector and it is equivalent to a subrepresentation of the left-regular representation. 
 	\end{lemma}
 	\begin{proof}
  Notice that when $\pi$ has a Parseval frame vector $v$, the analysis operator is an isometry and $\theta(\pi_gv)=\theta\theta^*(\delta_g)$ for all $g\in G$. Since $\theta$ is an isometry, the operator $\theta\theta^*$ is the projection onto the range of $\theta$. Thus $\{\theta\pi_gv\}$ is a projection of an orthonormal basis, and therefore it is a Parseval frame for its span. Since $\theta$ intertwines $\pi$ and $\lambda$ (Lemma \ref{commutation}), it follows that the projection $\theta\theta^*$ commutes with $\lambda$ and so $\pi$ is equivalent to the subrepresentation of $\lambda$ on the range of the projection $\theta\theta^*$, and $\{\pi_g v\}$ is equivalent to the Parseval frame $\{\theta\theta^*\delta_g=\lambda_g\theta\theta^*\delta_e\}$.  See \cite[Theorem 3.11, Proposition 6.2]{HL}.
\\  
\par  In the case when $v$ is just a frame vector one can use \cite[Theorem 5.3.4]{Chr03}: 
\begin{theorem}\label{thchr} Let $\{f_k\}_k$ be a frame for $H$ with frame operator $S:=\theta^*\theta : H \to H$, where $\theta$ is the analysis operator. Denote the positive square root of $S^{-1}$ by $S^{-1/2}$. Then $\{S^{-1/2}f_k \}_k$ is a Parseval frame. 
\end{theorem}
 Now, if $\{\pi_g(w)\}_g$ is a frame then $\{S^{-1/2}\pi_g(w) \}_g$ is a Parseval frame.  From Lemma \ref{commutation} and functional calculus, $S^{-1/2}$ commutes with $\pi$, hence $S^{-1/2}w$ is a Parseval frame vector for $\pi$, and as above we get that $\pi$ is a subrepresentation of $\lambda$. 
\end{proof}

\par Returning to the proof of (iv)$\Rightarrow$(i), with Lemma \ref{lemsub}, we get $1_G\prec\pi<\lambda$ and so $G$ is amenable.
The next lemma shows that item (v) also implies the containment $1_G\prec \pi < \lambda$.
\newcommand{\Range}{\operatorname*{Range}}
\newcommand{\Ker}{\operatorname*{Ker}}
\begin{lemma}\label{lemaP1} Suppose $\pi$ has a weak frame vector $\psi\in H$. Then $\pi$ is equivalent to a subrepresentation of the left regular representation.
\end{lemma}
\begin{proof} We use the polar decomposition for $\theta$, the analysis operator. The operator $\theta$ is bounded, since the vector $\psi$ is a Bessel vector, and $\theta$ is injective, since $\psi$ is total. Also, according to Lemma \ref{commutation}, $\theta$ intertwines the representations $\pi$ on $H$ and $\lambda$ on $l^2(G)$. 
	
	We do the polar decomposition of $\theta$: $\theta=U|\theta|$ where $|\theta|=(\theta^*\theta)^{1/2}$ and $U$ is a partial isometry with initial space $\cj{\Range|\theta|}=\Ker|\theta|^\perp$; moreover, since $\theta$ is injective, so is $|\theta|$, and $U$ is an isometry. 
	
	We claim that $U$ intertwines $\pi$ and $\lambda$. We know that $\theta^*\theta$ commutes with $\pi$ and therefore $|\theta|=(\theta^*\theta)^{1/2}$ commutes with $\pi$ as well. 
	
	We have for $g\in G$ and $x\in H$,
	$$U\pi_g(|\theta|x)=U|\theta|\pi_gx=\theta\pi_gx=\pi_g\theta x=\pi_gU(|\theta|x).$$
	Since $\cj{\Range|\theta|}=\Ker|\theta|^\perp$ is the initial space of $U$ and this is the whole space $H$, we obtain that 
	$$U\pi_g=\pi_gU\mbox{ for all }g\in G.$$
	 Since $U$ is an isometry, it follows that $\pi$ is equivalent to a subrepresentation of $\lambda$.
\end{proof}

\par In the proof of the final step, we will see that (vi) implies the weak containments $1\prec \pi \prec \lambda $. To prove that (vi) implies (i), we can use again the polar decomposition, now for unbounded operators (\cite[Theorem 7.20]{Wei}). However, it is more instructive to go through a different approach that uses the Spectral Theorem. 	We will need some lemmas which are interesting on their own.
 	
 	\begin{lemma}\label{lema}
 		
 		Suppose $\pi$ and $\rho$ are unitary representations of the group $G$ on the Hilbert spaces $H$ and $K$ respectively. Suppose there exists an increasing sequence of projections $P_n$ in the commutant $\pi'$, $n\in\bn$ such that the subrepresentation $\pi|_{P_n}$ is weakly contained in $\rho$ for all $n\in\bn$, and that $\cj{\cup_nP_nH}=H$. Then $\pi\prec\rho$. 
 	\end{lemma}
 	
 	\begin{proof}
 		Let $v\in H$, $F\subset G$ finite, and $\epsilon>0$. Then there exists $n\in\bn$ such that 
 		$$\left|\ip{\pi_gv}{v}-\ip{\pi_gP_nv}{P_nv}\right|<\epsilon/2\mbox{ for all }g\in F;$$
 		this is because $P_nv\rightarrow v$ and so $\pi_gP_nv\rightarrow \pi_gv$ for all $g\in F$.
 		
 		Now $\pi|_{P_n}\prec \rho$, so there exist vectors $u_1,\dots,u_p\in K$ such that 
 		$$\left|\ip{\pi_gP_nv}{P_nv}-\sum_{i=1}^p\ip{\rho_gu_i}{u_i}\right|<\epsilon/2\mbox{ for all }g\in F.$$
 		
 		Then, by the triangle inequality, 
 		$$\left|\ip{\pi_gv}{v}-\sum_{i=1}^p\ip{\rho_gu_i}{u_i}\right|$$$$\leq \left|\ip{\pi_gv}{v}-\ip{\pi_gP_nv}{P_nv}\right|+\left|\ip{\pi_gP_nv}{P_nv}-\sum_{i=1}^p\ip{\rho_gu_i}{u_i}\right|$$$$<\epsilon/2+\epsilon/2=\epsilon,$$
 		and this shows that $\pi\prec\rho$. 
 	\end{proof}
 	
The next lemma is based on the definition of weak inclusion and is straightforward. We provide the proof for the benefit of the reader. 
 	\begin{lemma}
 		\label{lemb1}
 		If $\pi$ is a unitary representation of $G$ on the Hilbert space $H$ and $n\in\bn\cup\{\infty\}$, then  $\pi^n=\underbrace{\pi\oplus\pi\oplus\dots\oplus\pi}_{n\mbox{ times}}\prec\pi$.
 	\end{lemma}
 	
 	\begin{proof}
		For $n=2$, let $v_1\oplus v_2\in H\oplus H$, $F\subset G$ finite, and $\epsilon>0$. Then, for all $g\in F$. 
		$$\left|\ip{(\pi\oplus\pi)_g(v_1\oplus v_2)}{v_1\oplus v_2}-\left(\ip{\pi_gv_1}{v_1}+\ip{\pi_gv_2}{v_2}\right)\right|=0<\epsilon.$$
		This implies that $\pi\oplus\pi\prec\pi$.

		Next, for $n\in\bn$, by induction,
		$$\pi^{2^{k+1}}=\pi^{2^k}\oplus\pi^{2^k}\prec\pi^{2^k}\prec\pi.$$
		Then for $k$ such that $n\leq 2^k$, we have $\pi^n<\pi^{2^k}\prec \pi$ and therefore $\pi^n\prec\pi$.

In the case when $n=\infty$, let $P_n$ be the projection from $H^\infty$ to $H^n\oplus 0\oplus 0\dots$. Then $\cup P_nH^\infty$ is dense in $H^\infty$, and $\pi^\infty|_{P_n}\simeq \pi^n\prec\pi$. 
 Using Lemma \ref{lema}, we obtain that $\pi^\infty\prec\pi$.
 	\end{proof}

Assume now (vi). The analysis operator $\theta$, has domain 
$$\mathscr D(\theta):=\{v\in H: \sum_{g\in G} |\ip{v}{\pi_g\psi}|^2<\infty\},$$
$$\theta(v)=\left(\ip{v}{\pi_g\psi}\right)_{g\in G}.$$
The hypotheses imply that $\theta$ is densely defined and injective. We show that $\theta$ is also closed. Indeed, if $v_n\to v$ in $H$ and $\theta(v_n)\to(\alpha_g)_{g\in G}$ in $l^2(G)$, then, for any $g\in G$, the sequence $\{\ip{v_n}{\pi_g\psi}\}$ converges  both to $\alpha_g$ and to $\ip{v}{\pi_g\psi}$. Therefore $\ip{v}{\pi_g\psi}=\alpha_g$, for $g\in G$. This means that $v\in\mathscr D(\theta)$ and $\theta v=(\alpha_g)_{g\in G}$.

With \cite[Theorem 5.39]{Wei} the operator $\theta^*\theta$ is self-adjoint and injective, and $\mathscr D(\theta^*\theta)$ is a core of the operator $\theta$ in the sense that the closure of the graph of $\theta|_{\mathscr D(\theta^*\theta)}$ is the graph of $\theta$. In particular $\mathscr D(\theta^*\theta)$ is dense in $\mathscr D(\theta)$ with the inner product $\ip{v_1}{v_2}_\theta=\ip{v_1}{v_2}+\ip{\theta v_1}{\theta v_2}$.

In addition, the same computation as in Lemma \ref{commutation}, shows that, if $v\in\mathscr D(\theta)$ then $\pi_gv\in\mathscr D(\theta)$ and $\theta\pi_gv=\lambda(g)\theta v$, for all $g\in G$. And therefore, if $v\in \mathscr D(\theta^*\theta)$ then $\pi_gv\in \mathscr D(\theta^*\theta)$ and $\theta^*\theta\pi_gv=\pi_g\theta^*\theta v$, for all $g\in G$.

Let $E:\mathcal B(\br)\rightarrow \operatorname*{Projections}(H)$ be the spectral measure of $\theta^*\theta$, where $\mathcal B(\br)$ is the Borel sigma-algebra on $\br$. Let $\Delta_n=[1/n,n]$. Since $\theta^*\theta$ is a non-negative operator and it is injective $E(\{0\})=0$ so 
$$\cup_n \Delta_n=(0,\infty)\supseteq \operatorname*{Spectrum}(\theta^*\theta).$$
This means that the span of the union of the subspaces $E(\Delta_n)H$ is the whole space $H$.

We will show that $E(\Delta_n)\psi$ is a frame vector for the subrepresentation $\pi|_{E(\Delta_n)}$, which will imply that this subrepresentation is equivalent to a subrepresentation of $\lambda$ (Lemma \ref{lemsub}), and since the subspaces $E(\Delta_n)H$ converge to the whole space $H$, Lemma \ref{lema} will give us (i). 

We prove that, for $v\in H$ and $n\in\bn$, $E(\Delta_n)v\in \operatorname*{Domain}(\theta^*\theta)$  and 
\begin{equation}
	\label{eqvi1}
	\frac1{n^2}\left\|E(\Delta_n)v\right\|^2\leq\left\|\theta E(\Delta_n)v\right\|^2\leq n^2\left\|E(\Delta_n)v\right\|^2.
\end{equation}

We know that a vector $h\in H$ is in the domain of $\theta^*\theta$ iff 
$$\int_{\br}|x|^2\,d\ip{E(x)h}{h}<\infty.$$
We have 
$$\int_{\br}|x|^2\,d\ip{E(x)E(\Delta_n)v}{E(\Delta_n)v}=\int_{\br}|x|^2\chi_{\Delta_n}(x)\,d\ip{E(x)v}{v}\leq n^2\|v\|^2<\infty.$$
Thus $E(\Delta_n)v$ is in the domain of $\theta^*\theta$.

Next we compute
$$\|\theta E(\Delta_n)v\|^2=\ip{\theta E(\Delta_n)v}{\theta E(\Delta_n)v}=\ip{\theta^*\theta E(\Delta_n)v}{E(\Delta_n)v}$$$$=\int_{\br}|x|^2\chi_{\Delta_n}(x)\,d\ip{E(x)v}{v}.$$
Since $1/n\leq |x|\leq n$ for $x\in \Delta_n$ we get that
$$\frac1{n^2}\int_{\br}\chi_{\Delta_n}(x)\,d\ip{E(x)v}{v}\leq \|\theta E(\Delta_n)v\|^2\leq {n^2}\int_{\br}\chi_{\Delta_n}(x)\,d\ip{E(x)v}{v}.$$
On the other hand 
$$\int_{\br}\chi_{\Delta_n}(x)\,d\ip{E(x)v}{v}=\ip{E(\Delta_n)v}{v}=\|E(\Delta_n)v\|^2.$$
Thus we obtain \eqref{eqvi1}.

Also,
$$\|\theta E(\Delta_n)v\|^2=\sum_{g\in G}\left|\ip{E(\Delta_n)v}{\pi_g\psi}\right|^2=\sum_{g\in G}\left|\ip{E(\Delta_n)v}{E(\Delta_n)\pi_g\psi}\right|^2$$
$$=\sum_{g\in G}\left|\ip{E(\Delta_n)v}{\pi_gE(\Delta_n)\psi}\right|^2,$$
because $\pi_g$ commutes with $\theta^*\theta$ and therefore it commutes with $E(\Delta_n)$. 

This means, together with \eqref{eqvi1}, that $E(\Delta_n)\psi$ is a frame vector for $\pi|_{E(\Delta_n)}$. Therefore, with Lemma \ref{lemsub}, $\pi|_{E(\Delta_n)}$ is equivalent to a subrepresentation of $\lambda$.

We also have that $\cj{\cup_nE(\Delta_n)H}=H$. With Lemma \ref{lema} we get that $\pi\prec \lambda$. Thus (vi) implies (i).

\begin{remark} Using Theorem \ref{amen}, we can infer that a countable group $G$ is nonamenable iff any representation which has almost invariant vectors does not have a weak frame vector.  
 As a consequence, Theorem 3.6 from \cite{GP} is greatly improved. Moreover, we can do away with some terminology introduced in \cite{GP} (e.g., property NFT now simply means nonamenable).  
\end{remark}

\begin{example}\label{abg} As mentioned above, if $G$ is abelian, the Markov-Kakutani fixed point theorem supplies a left invariant state on $l^{\infty}(G)$, and therefore $G$ is amenable. When $G$ is countable, F{\o}lner criterion can be applied. We provide another argument after Theorem \ref{thai1}. Recall that if $G$ is an abelian, countable, discrete group then its dual 
$$\widehat{G}=\{\xi : G\to\mathbb{T} \text{  }| \text{  } \xi\text{ is a group homeomorphism } \}$$ is a compact, abelian, topological group such that its normalized Haar measure $\hat{\mu}$ is a Borel probability measure. The topology on $\widehat{G}$ is the topology of compact convergence.

The elements $h$ in $G$ can be seen as functions on the dual  $h: \widehat{G}\ni \xi\mapsto \overline{\xi(h)}$, and they form an orthonormal basis for $L^2(\widehat{G})$.

 The Fourier transform is defined for $f\in l^2(G)$ by 
$$ \mathcal{F} (f) (\xi) = \sum_{h\in G} f(h)\overline{ \xi (h) },$$
and is an isometric isomorphism $ \mathcal{F} : l^2(G)\to L^2(\widehat{G})$, by Plancherel's Theorem (for more details we refer to \cite{Folland}).

Define the unitary representation $\pi : G\to \mathcal U(L^2(\widehat{G}))$, by the multiplication operators:
$$\pi_g (f)(\xi) =  \overline{ \xi (g)} f(\xi), \quad(g\in G,\xi\in \widehat{G}, f\in L^2(\widehat{G})).$$ 
Notice that the constant function  $f=1\in L^2(\widehat{G})$ is an ONB vector.     
The Fourier transform intertwines the left-regular representation $\lambda$ on $l^2(G)$ and $\pi$ on $L^2(\widehat{G})$; in other words, the Fourier transform changes translations into multiplications. Indeed, for $f\in l^2(G)$, $g\in G$ and $\xi\in \widehat{G}$, we have
$$\mathcal F\lambda_g f(\xi)=\sum_{h\in G}\lambda_gf(h)\overline{\xi(h)}=\sum_{h\in G}f(g^{-1}h)\overline{\xi(h)}=\sum_{h'\in G}f(h')\overline{\xi(gh')}=\overline{\xi(g)}\sum_{h\in G}f(h)\overline{\xi(h)}=\pi_g\mathcal F f(\xi).$$
Then, after Theorem \ref{thai1} we will see that this representation has almost invariant vectors, and having ONB vector $f=1$,  by item ii) of Theorem \ref{amen} we obtain that $G$ is amenable. Of course, one might also finish off with its left regular representation having almost invariant vectors: the Fourier transform will bring it there, using the intertwining above. That intertwining however, works because the maps $h: \widehat{G}\ni \xi\mapsto \overline{\xi(h)}$ form an ONB in $L^2(\widehat{G})$ i.e.,  $\{\pi_g (1)\}_{g\in G}$ is ONB in $L^2(\widehat{G})$, i.e. $f=1$ is ONB vector for $\pi$. 
\end{example}
We are interested in the question: when does a representation, in particular a frame representation of a group, have almost invariant vectors?
In the case of a countable, abelian group $G$, we can use the Spectral Theorem to classify all such representations. First we need some notation:
for a regular Borel measure $\nu$ on $\widehat{G}$, define the representation $\pi_\nu$ of $G$ on $L^2(\widehat{G},\nu)$, by multiplication by the character $g$: 
\begin{equation}
	\label{eqai1}
	(\pi_\nu(g)f)(\xi)=\overline{\xi(g)}f(\xi),\quad(\xi\in\widehat{G},f\in L^2(\widehat{G},\nu)).
\end{equation}

For a cardinality $n$ and a representation $\pi$ of $G$ on a Hilbert space $H$, recall the notation 
\begin{equation}
	\label{eqai2}
	\pi^{n}=\underbrace{\pi\oplus\pi\oplus\dots\oplus\pi}_{n\mbox{ times}}.
\end{equation}

The Spectral Theorem states that any unitary representation of a group $G$ on a separable Hilbert space is equivalent to a representation of the form 
$$\pi=\pi_{\nu_\infty}^{\infty}\oplus\pi_{\nu_1}\oplus\pi_{\nu_2}^{2}\oplus\pi_{\nu_3}^{3}\oplus\dots,$$
where $\nu_\infty,\nu_1,\nu_2,\nu_3,\dots$ are mutually singular measures on $\widehat{G}$.

Since multiplicity does not affect weak equivalence of representations (Lemma \ref{lemb1}) we have that this representation is weakly equivalent to
$$\pi\sim \pi_{\nu_\infty}\oplus\pi_{\nu_1}\oplus\pi_{\nu_2}\oplus\pi_{\nu_3}\oplus\dots.$$
Since the measures $\{\nu_i\}$ are mutually singular, they are supported on some mutually disjoint sets $\{E_i\}$. Then the representation $\pi_{\nu_\infty}\oplus\pi_{\nu_1}\oplus\pi_{\nu_2}\oplus\pi_{\nu_3}\oplus\dots$ is equivalent to the representation $\pi_{\nu}$ where $\nu=\nu_\infty+\nu_1+\nu_2+\dots$, i.e., 
$$\nu(E)=\nu_\infty(E\cap E_\infty)+\nu_1(E\cap E_1)+\nu_2(E\cap E_2)+\dots,\quad(E\mbox{ Borel set in }\widehat{G}).$$
Hence, to study the existence of almost invariant vectors for representations of an abelian group $G$, it is enough to focus on representations of the form $\pi_\nu$. 

We define the {\it support} $\supp\nu$ of a regular Borel measure $\nu$ on $\widehat{G}$ to be the smallest compact set of full measure, or, equivalently, the complement of the largest open set of measure zero. In this way, $x\in \supp\nu$  if and only if for any neighborhood $V$ of $x$, $\nu(V\cap\supp \nu)>0$. 
\newcommand{\id}{\operatorname*{id}}
\begin{theorem}\label{thai1}
	Let $G$ be countable, abelian and let $\nu$ be a regular Borel measure on $\widehat{G}$. The representation $\pi_\nu$ has almost invariant vectors if and only if the identity element $\id_{\widehat{G}}$ is in the support of $\nu$. 
\end{theorem}

\begin{proof}
	Assume $\id_{\widehat{G}}$ is in the support of $\nu$. If $\nu(\{\id_{\widehat{G}}\})>0$, then we can take the vector $v=\frac1{\sqrt{\nu(\{\id_{\widehat{G}}\})}}\chi_{\{\id_{\widehat{G}}\}}$, and this is an invariant vector for $\pi_\nu$. Therefore, we can assume $\nu(\{\id_{\widehat{G}}\})=0$.

	Consider $\{V_n\}_{n\in\bn}$ some decreasing basis of compact neighborhoods of $\id_{\widehat{G}}$ in $\hat G$, so $\cap_n V_n=\{\id_{\widehat{G}}\}$. In particular $\lim_n\nu(V_n)=0$. We also have that $\nu(V_n)>0$ because $\id_{\widehat{G}}\in\supp\nu$. 
	
	Define the vectors $v_n=\frac{1}{\sqrt{\nu(V_n)}}\chi_{V_n}$ in $L^2(\widehat{G},\nu)$. We will show that $\{v_n\}$ is a sequence of almost invariant vectors. Let $g\in G$; we have:
	$$\|\pi_\nu(g)v_n-v_n\|^2=\frac{1}{\nu(V_n)}\int_{\widehat{G}}|\overline{\xi(g)}-1|^2\chi_{V_n}(\xi)\,d\nu(\xi)=\frac{1}{\nu(V_n)}\int_{V_n}|\overline{\xi(g)}-1|^2\,d\nu(\xi).$$
	The map $\xi\mapsto \overline{\xi(g)}$ is continuous on $\widehat{G}$, therefore, given $\epsilon>0$, there exists an index $N$ such that $|\overline{\xi(g)}-1|<\epsilon$ for all $\xi\in V_n$, and all $n\geq N$. Then 
	$$\|\pi_\nu(g)v_n-v_n\|^2\leq \frac1{\nu(V_n)}\epsilon^2\nu(V_n)=\epsilon^2,\quad (n\geq N),$$
	and this means that $\{v_n\}$ are almost invariant vectors.

	For the converse, if $\id_{\widehat{G}}$ is not in the support of $\nu$, then for any $\xi\in\supp\nu$, since $\xi\neq \id_{\widehat{G}}$, there exists $g_\xi\in G$ such that $\xi(g_\xi)\neq 1$. Then, by continuity, there exists a neighborhood $V_\xi$ of $\xi$ in $\widehat{G}$, and a constant $c_\xi>0$ such that $|\eta(g_\xi)-1|\geq c_\xi$ for all $\eta\in V_\xi$. 
	
	Since $\supp\nu$ is compact, there is a finite subcover $V_{\xi_1},\dots,V_{\xi_n}$ of $\supp\nu$. Suppose now $\{v_i\}$ is some net of almost invariant  vectors, $\|v_i\|=1$ for all $i$. To simplify notation, let $g_k:=g_{\xi_k}$, with corresponding constants $c_k$, $k=1,\dots,n$. Consider the sum
	$$S_i:=\sum_{k=1}^n\|\pi(g_k)v_i-v_i\|^2.$$
	On one hand, the sum converges to zero, because the vectors $\{v_i\}$ are almost invariant. On the other hand
	$$S_i=\sum_{k=1}^n\int_{\widehat{G}}|\overline{\xi(g_k)}-1|^2|v_i(\xi)|^2\,d\nu(\xi)\geq \sum_{k=1}^n\int_{V_{\xi_k}}|  \overline{\xi(g_k)}-1|^2|v_i(\xi)|^2\,d\nu(\xi)$$
	$$\geq\sum_{k=1}^n c_{k}\int_{V_{\xi_k}}|v_i(\xi)|^2\,d\nu(\xi)\geq \min_k c_{k}\sum_{k=1}^n\int_{V_{\xi_k}}|v_i(\xi)|^2\,d\nu(\xi)$$
	$$\geq \min_k c_{k}\int_{\cup_k V_k}|v_i(\xi)|^2\,d\nu(\xi)\geq \min_k c_{k}\int_{\supp\nu}|v_i(\xi)|^2\,d\nu(\xi)= \min_k c_{k},$$
	a contradiction. 
\end{proof}

\begin{remark}
	\label{remai2} If $G$ countable and abelian then the support of the normalized Haar measure $\hat{\mu}$ is the whole $\widehat{G}$. Obviously, $\id_{\widehat G}$ belongs to  the support of $\hat{\mu}$, thus by Theorem \ref{thai1} the representation from Example \ref{abg} has almost invariant vectors.  
		
\end{remark}

\begin{example}
Once abelian groups fall in the class of amenable groups, one can use a variety of algebraic operations to extend the class. 
We consider for a  moment the Baumslag-Solitar group $BS(1,2)=\{u,t \mid utu^{-1}=t^2\}$, which is amenable because  it is {\it metabelian}, i.e., a group whose commutator subgroup is abelian. Equivalently, a group $G$ is metabelian if and only if there is an abelian normal subgroup $A$ such that the quotient group $G/A$ is abelian.

 Actually, for $N\in\mathbb{N}$ all groups $BS(1,N)=\{u,t \mid utu^{-1}=t^N\}$ are metabelian. This can be seen from the identification of $BS(1,N)$ with the semidirect product $B \rtimes_{\alpha} \bz$ where $$B=\mathbb{Z}[1/N]=\{k/N^n : k\in\bz, n\in\bn\},$$ and $\alpha: B\to B$, $\alpha(b)=Nb$ implementing the morphism $\mathbb{Z}\to\text{Aut}(B)$ by $j\to \alpha^j$. 
We focus on $BS(1,2)$ because many of its representations bear special connection to the first wavelets/frame constructions (see, e.g., \cite{HW}). For example, the unitary operators $U,T:L^2(\mathbb{R}) \to L^2(\mathbb{R})$:
\[
Uf(x)=\dfrac{1}{\sqrt{2}}f(x/2), \quad Tf(x)=f(x-1).
\]
define a representation  
\[
\pi: BS(1,2) \to \mathcal U(L^2(\mathbb{R})), \quad u \to U, \quad t \to T.
\]
 If we consider a {\it subset} of the whole group, namely $S= \{ u^mt^n \mid m,n, \in \mathbb{Z} \} \subset BS(1,2)$, then there is a  vector $\psi \in L^2(\mathbb{R})$ such that $\{\pi_g\psi\}_{g \in S}$ is in fact an ONB. Actually there are many such vectors, the so called wavelets, the simplest one being the Haar wavelet $\psi= \frac1{\sqrt{2}}\left(\chi_{[0,1/2)}-\chi_{[1/2,1)}\right)$ (see e.g., \cite{HW}). We caution the reader that many authors use the term `` frame/wavelet representation'' for representations of $BS(1,N)$ and certain generalizations (see \cite{HL,DJ06,DJ07,LPT}), to mean that only the wavelet part $u^mt^n $ (translations dilations) is applied to the vector, not all elements of the group.

This representation does not admit a frame vector for the {\it whole} group $BS(1,2)$; actually, not even a Bessel sequence. Assume by contradiction that $\psi\in L^2(\br)$ is a Bessel vector for the whole group, so $\{\pi_g\psi : g\in BS(1,2)\}$ is a Bessel sequence for $L^2(\br)$ with Bessel bound $B>0$. Note that, for $n\in\bn$, $U^{-n}TU^{n}$ is the operator of translation by $\frac1{2^n}$ in $L^2(\br)$ which we denote by $T_{1/2^n}$. Indeed, for $f\in L^2(\br)$ and $x\in\br$,
$$U^{-n}TU^nf(x)=\sqrt{2^n} TU^nf(2^nx)=\sqrt{2^n}U^nf(2^nx-1)$$
$$=f\left(\frac1{2^n}(2^nx-1)\right)=f\left(x-\frac1{2^n}\right).$$

But then
$$\pi(u^{-n}tu^n)\psi=U^{-n}TU^{n}\psi=T_{1/2^n}\psi\rightarrow\psi\mbox{ as }n\to\infty.$$
Therefore, there exists $N\in\bn$ such that, for $n\geq N$,
$$ |\ip{\pi(u^{-n}tu^n)\psi}{\psi}|\geq \ip{\psi}{\psi}-\frac12\|\psi\|^2=\frac12\|\psi\|^2.$$
Then, with the upper frame bound, $$B\|\psi\|^2\geq\sum_{g\in BS(1,2)}|\ip{\pi_g\psi}{\psi}|^2\geq \sum_{n\geq N}|\ip{\pi(u^{-n}tu^n)\psi}{\psi}|^2\geq \sum_{n\geq N}\frac12\|\psi\|^2=\infty,$$
a contradiction. 
\end{example}

	 We turn now to groups with {\it infinite conjugacy classes (ICC)}, that is for all $g\in G$, $g\neq e$, the conjugacy class $\{xgx^{-1} : x\in G\}$ is infinite. 
	 
	 \begin{theorem}
	 	\label{thai3}
	 	Let $G$ be an amenable ICC group. Then any frame representation has almost invariant vectors. 
	 \end{theorem}

	 \begin{proof}
	 	Let $\pi$ be a frame representation. By Lemma \ref{lemsub}, $\pi$ is equivalent to a subrepresentation of the left-regular representation $\lambda$, so there exists a projection $P$ in the commutant $\lambda'$ such that $\pi$ is equivalent to $\lambda|_P$. 
	 	\newcommand{\tr}{\operatorname*{Tr}}
	 	Since the group $G$ is ICC, the von Neumann algebra $\lambda'$ is a $II_1$ factor with trace $\tr(x)=\ip{x\delta_e}{\delta_e}$, $x\in\lambda'$. The projection $P$ has a trace $\tr(P)=\ip{P\delta_e}{\delta_e}>0$. Let $n\in\bn$ with $\frac1n<\tr(P)$. Since $\lambda'$ is a $II_1$ factor, there exists a projection $Q$ in $\lambda'$ such that $\tr(Q)=\frac1n$ and $Q\leq P$. Using again the $II_1$ factor property of $\lambda'$ we can construct inductively projections $Q_1=Q,Q_2,\dots, Q_n$ in $\lambda'$ such that $Q_i\perp Q_j$ for $i\neq j$, $\tr(Q_i)=\frac1n$, and $Q_1+Q_2+\dots+Q_n=I$. $\tr(Q_i)=\frac1n$  means that the projections are equivalent in $\lambda'$, i.e., there exist partial isometries $V_{i}$ in $\lambda'$ with $V_i^*V_i=Q$ and $V_iV_i^*=Q_i$. This implies that the representations $\lambda|_Q$ and $\lambda|_{Q_i}$ are equivalent. Then the representations $\lambda|_Q^{(n)}$ and $\lambda|_{Q_1}\oplus\lambda|_{Q_2} \oplus\dots\lambda|_{Q_n}=\lambda$ are equivalent. 

	 	Since $G$ is amenable, $\lambda$ has almost invariant vectors. Since weak containment does not depend on multiplicity (Lemma \ref{lemb1}), $\lambda|_Q$ has almost invariant vectors, and therefore $\lambda|_P$ has as well. This shows that $\pi$ has almost invariant vectors. 
	 \end{proof}

\section{Weak Frame Representations}\label{wfr}
Given a group representation we seek to relax the inequalites \eqref{frame} while maintaining the requirements about the almost invariant vectors. We are thus weakening the amenability property. We mention that there are many weaker forms of amenability which can be considered in a more general setting of a locally compact group (for example weakly amenable \cite{CoHa}, inner amenable \cite{Eff}, Haagerup's property \cite{Cherix}, relatively amenable \cite{CaMo}) which are defined with the Fourier algebra of the group and the left regular representation in the background, or by modifying the left invariance of the $l^{\infty}(G)$ state, or by relaxing a fixed point property of the action of the group on certain topological spaces. Our attempt is based on Theorem \ref{amen} by simply relaxing the frame inequalities.  
Recall that a collection of vectors $\{f_n\}_{n \in S}$ ($S$ countable) in a Hilbert space $H$ is called a {\it weak frame} if for every $v \in H$, $v \neq 0$, we have: 
\begin{equation} \label{weak frame}
0 < \sum_{n\in S} |\langle v,f_n\rangle|^2 <\infty
\end{equation}

\begin{definition}\label{afrem} 
Let $G$ be a countable discrete group. We say that  $G$ is {\it framenable} if there exist a unitary representation $\pi$ of $G$ with $1_G\prec \pi$, a countable subset $S\subset G$, and a vector $v\in H$ such that $\{\pi_gv\}_{g \in S}$ is a weak frame.
\end{definition}
\begin{remark}\label{tri} The trivial representation of any countable group, obviously, does not admit a weak frame vector (nor a Bessel vector, actually). More generally, if a representation has a fixed point then there cannot be a weak frame vector.  The definition above emphasizes a degree of separation between a representation $\pi$ with weak frame vectors and the trivial representation. 
We also note that as far as weakening the frame condition goes, while maintaining $1_G\prec \pi$, one cannot relax it any further to just a Bessel vector, because the condition becomes trivial. Indeed, for any countable group the representation $\pi:=1_G \oplus \lambda_G$ satisfies $1_G\prec\pi$. Also, with $v:=(0, \delta_e)$ we have that $\{\pi_g(v)\}_{g\in G}$ forms a Bessel sequence. This shows that any group has a representation $\pi$ with $1_G\prec \pi$ which has a Bessel vector. Hence, the requirement that  $\{\pi_g(v)\}_{g\in S}$ be total is crucial in the definition of framenability.  
 \end{remark}
  
By Theorem \ref{amen} amenable groups are framenable. The latter class however, is very large as we shall see later, based on some permanence properties which are not shared  by the amenable class. For example, amenability passes to subgroups, e.g., a sequence of almost invariant vectors for the left regular representation of a group stays almost invariant for its restriction to a subgroup. While this property is maintained with respect to arbitrary representations, the (weak) frame vector may not transfer to the subgroup representation. The Bessel inequality does, but the totality condition can be lost.  One of the first important facts is that amenability is an antagonist of property (T) for infinite groups. Actually, property (T) is a strong antagonist of a weaker form of amenability known as Haagerup’s property.

\begin{definition} $G$ satisfies {\it Haagerup's property} if  there exists a representation $\pi: G\to \mathcal U(H)$ such that $1_G\prec \pi$ and 
$\displaystyle{\lim_{g\to\infty } \langle \pi_g \xi, \eta   \rangle=0}$,  for all $\xi, \eta \in H$. The limit over $g\to\infty$ means that $\pi$ is $c_0$, i.e., for all $\epsilon >0$  there exists $E\subset G$ finite such that for every $g\notin E$ we have $| \langle \pi_g \xi, \eta   \rangle  |<\epsilon$.  

\end{definition}
\begin{remark} Amenable groups satisfy Haagerup's property. This follows from multiple characterizations of both amenability and Haagerup's property (see e.g., \cite{BO}). Notice that the $c_0$ condition of a representation is satisfied whenever $\pi$ admits a (weak) frame vector $w$. By density, one needs to check the limit is zero for all $\xi=\pi_g(w)$. Owing to the Bessel inequality, one concludes by the divergence test that the general term in the series must converge to zero. Hence, all other things being equal (i.e., $1_G\prec \pi$) from Theorem \ref{amen} we see (again) that amenable implies Haagerup's property. Let us emphasize that the situation is quite different when the (weak) frame is taken with respect to a (countable) subset $S\subset G$. The divergence test in this case does not warrant $\displaystyle{\lim_{g\to\infty } \langle \pi_g \xi, \eta   \rangle=0}$ because $g$ can approach $\infty$ outside $S$. We will see that there are many examples of framenable groups without Haagerup's property, but we are not aware of the other way around (see the list of examples and problems in the last section).  

\end{remark}

\begin{definition}
A countable group $G$ {\it satisfies property $(T)$} if any unitary representation  $\pi:G\to \mathcal U(H)$  with $1_G\prec\pi$ has a nonzero $G$-invariant vector $w$, i.e., $\pi_g w = w$, for all $g\in G$. In short, a group satisfies property (T) if whenever $1_G\prec \pi$ then $1_G < \pi$.  
If $N$ is a subgroup of $G$, the pair $(N,G)$ {\it satisfies relative property (T)} if any unitary representation  $\pi:G\to \mathcal U(H)$  with $1_G\prec\pi$ has a nonzero $N$-invariant vector.

\end{definition}

\begin{remark} 
A countable group with Haagerup's property cannot admit a countable subgroup $N$ such that the pair $(N,G)$ satisfies relative property (T) (a non zero invariant vector violates $\displaystyle{\lim_{g\to\infty } \langle \pi_g \xi, \eta   \rangle=0}$).    
\end{remark}

Non framenability may seem difficult to check for a given group. Actually, all our examples arise from the next proposition which shows that the restrictive class of property (T) groups is not framenable. It would be interesting to find non property (T) groups which are non framenable (see the problem list in the next section). 

\begin{proposition}\label{nonT} If $G$ is framenable then $G$ cannot satisfy property (T).
\end{proposition}
\begin{proof}
 Let $S\subset G$ a countable subset and $\pi$ be a representation having almost invariant vectors such that there is $v \in H$ with:
\[
0 < \sum_{g \in S} |\langle x,\pi_gv \rangle |^2 < \infty, \quad \text{ for all }x\in H, x\neq 0
\]
If $G$ had property (T), then there would be a $w \neq 0$ with $\pi_gw=w$ for all $g \in G$. Then:
\[
 0 <\sum_{g \in S}|\ip{w}{\pi_gv}|^2= \sum_{g \in S} |\langle \pi_g w,\pi_gv \rangle |^2 =  \sum_{g \in S} |\langle w,v \rangle |^2 <\infty.
\]
The second inequality implies that $\langle w,v \rangle = 0$, which contradicts the first.

\end{proof}

The next sequence of propositions shows that the class of framenable groups is large. However, some permanence properties are not maintained anymore as for the class of amenable groups,  e.g., inheritance. These results will also be used in the next Examples section.   

\begin{proposition}\label{surj} If $N$ is a framenable group, $G$ a countable group, and $\varphi: G\to N$ a surjective morphism, then $G$ is framenable.
\end{proposition}
\begin{proof} 
Let $\pi: N\to \mathcal U(H)$ be a unitary representation of $N$ with $1_N\prec \pi$, $S\subset N$ countable, and $w\in H$ satisfying $0< \sum_{n\in S}|\langle\pi_n( w), v \rangle|^2<\infty $ for all $v\neq 0$. Define the unitary representation $\rho: G\to \mathcal U(H)$, $\rho_g(v):=\pi_{\varphi(g)}(v)$, for all $v\in H$.  Let $(v_i)_{i\in I}$ be a net of almost invariant vectors for $\pi$. Then $\displaystyle{\lim_{i\in I}\| \rho_g(v_i) - v_i\|=\lim_{i\in I}\| \pi_{\varphi(g)}(v_i) - v_i\|=0}$ for all $g\in G$, hence $1_G\prec \rho$.  
Now, let $\{n_k |  k\in\mathbb{N}  \}$ be a countable enumeration of the set $S$. Because $\varphi $ is surjective, for each $n_k\in S$ one can choose precisely one $g_k\in \varphi^{-1}(n_k)$. Then the set $S':=\{ g_k | k\in\mathbb{N} \}\subset G$ is countable, and for every $v\in H$ we have
$$\sum_{g\in S'}|\langle\rho_{g}( w), v \rangle|^2 = \sum_{k\in \mathbb{N} }|\langle\pi_{n_k}( w), v \rangle|^2=\sum_{n\in S}|\langle\pi_n( w), v \rangle|^2$$
Hence $\{ \pi_g (w) \}_{g\in S'}$ is a weak frame. Therefore $G$ is framenable. 
\end{proof}

\begin{proposition}\label{products} (i) If $G$ is a framenable group and $N$ is any countable or finite group then the direct product $G\times N $ is framenable. \\
(ii) If $G$ is framenable and $\alpha: G \to \textup{Aut}(N)$ is a group morphism then the semi-direct product $G\ltimes N$ is framenable. \\
(iii) If $G$ is a framenable group and $N$ is any countable or finite group then the free product $G\star N$ is framenable.
\end{proposition}
\begin{proof}
(i) and (ii) both follow from Proposition \ref{surj} and the fact that either projection $p: G\times N \to G$ or $p: G\ltimes N \to G$ is a surjective morphism. Recall that the multiplication in the semi-direct product  $G\ltimes N $ is given by $(g_1, n_1) \cdot (g_2, n_2):=(g_1g_2,  n_1\alpha_{g_1}(n_2) )$, hence projecting onto $G$ gives a group morphism. For (iii) we again apply Proposition \ref{surj} with $\phi: G\star N\to G$ defined by $\phi(g_1n_1 g_2 n_2\dots g_kn_k) = g_1g_2\dots g_k$ for all $k\in\mathbb{N}$ and $g_i\in G$ and $n_i\in N$, $i=1,\dots,k$.  
\end{proof}

\begin{corollary}
Framenability does not pass to (countable) subgroups. Also, there exist framenable groups that do not have Haagerup's property. 
\end{corollary}
\begin{proof}
Take any framenable group $G$ and any countable, property (T) group $N$. Then $G\times N$ is framenable, but its subgroup $N$ is not, by Proposition \ref{nonT}.  Also, $G\times N$ cannot satisfy Haagerup property because $N$ would too, by inheritance.  
\end{proof}

\begin{definition} A group $G$ is called {\it locally indicable} if for any finitely generated subgroup $N\leq G$ there exists a surjective morphism $\varphi : N\to \mathbb{Z}$. 
\end{definition}
There are many examples of locally indicable groups. All torsion free, one relator groups are locally indicable \cite{ClRo}. Also, Thurston's Stability Theorem asserts that for the unit interval $I$ every  subgroup of $\textup{Diff}_{+}^1(I)$ is locally indicable.

The following corollary is a direct consequence of Proposition \ref{surj}. Recall, that the {\it abelianization} of a group $G$ is the quotient group $G/[G,G]$ where $[G, G]$ is the {\it commutator subgroup} generated by all commutators $[x,y]:=xyx^{-1}y^{-1}$, $x,y\in G $. 

\begin{corollary}\label{abl} (i) If a group $G$ has an infinite abelianization then $G$ is framenable.  \\
(ii) If $G$ is finitely generated and locally indicable then $G$ is framenable.
\end{corollary}
It would be interesting to extend (ii) above to the case of infinitely many generators, see the problem list in the next section. Although not hereditary, framenability enlarges the (countable) amenable class towards the non property (T) class. Because we do not know whether this enlargement precisely hits non property (T) (we conjecture that it does not), it is instructive to find some permanence properties which in turn produce more examples. One of the main results in this section, Theorem \ref{ind}, shows that, in certain cases, framenability passes from a subgroup to the group. For this we need {\it induced representations}, i.e., the main tool used to obtain a group representation out of a subgroup's. The interested reader may consult \cite{Folland} or \cite{BHV} on the subject of induced representations. For convenience  we include an equivalent construction, plus some properties of the induced representation when the groups are countable.

Assume first that $\Gamma$ and $\Lambda$ are two countable groups and that $\Gamma$ admits a measure preserving action on a measurable space $(X,\nu)$. We denote this action by  $\Gamma \curvearrowright (X,\nu)$, with automorphisms $X\ni x\mapsto x\cdot \gamma =\gamma(x)\in X$ satisfying $x\cdot \gamma_1\gamma_2= (x\cdot \gamma_1)\cdot \gamma_2 $, $\forall$ $\gamma_{1,2}\in\Gamma$, $x\in X$. We further assume that we have a {\it cocycle}
$\alpha: X\times \Gamma \to\Lambda$
 $$\alpha(x, \gamma_1\gamma_2)=\alpha(x, \gamma_1) \cdot \alpha(x\cdot \gamma_1, \gamma_2), \mbox{ for all }\gamma_{1},\gamma_2\in\Gamma \text{ and a.e. }x\in X.$$
Now, if $\pi : \Lambda\to\mathcal U(H)$ is a unitary representation of $\Lambda$ then 
$$\mathcal{H} := L^2(X, H) =\{ \varphi: X \to H\mid \int_X \| \varphi(x) \|^2 d\nu < \infty  \}$$ is a Hilbert space and 
$$\rho: \Gamma\to {U}(\mathcal{H}), \quad \rho_{\gamma}(\varphi) (x): =\pi[  \alpha(x,\gamma ) ] (\varphi( x\cdot\gamma)   ) $$ is the {\it induced representation}  $\textup{Ind}_{\Lambda}^{\Gamma}\pi$ with respect to $\alpha$. It is the cocycle identity which is needed to show that $\rho$ is a group representation, i.e., $\rho_{hg}=\rho_h \rho_g$, for $g, h \in \Gamma$. 
Our extension theorem will be based on the set-up and properties described in the example below.

\begin{example}\label{cocy}
Let $N\leq G$ be a countable subgroup of countable $G$ with $\pi: N\to\mathcal U(H)$ a unitary representation of the subgroup $N$. Denote by 
$X=G/N$ the set of right cosets, and $D=\{e, g_1, g_2,\dots \}$ a set of representatives. By $r: G/N\to D$ we denote the map that associates to a  right coset $x$ its unique representative in $D$. The action of $G$ on $X$ (with respect to the counting measure) is given by: for 
$g\in G$, $X\ni x=Nh\mapsto (Nh)g=N(hg)\in X $.  Now, for any $g\in G$ and $x\in G/N$ there exists a unique element $\alpha(x,g)$ in $N$ such that 
$\alpha (x,g)=r(x)\cdot g\cdot r(x g)^{-1}$. To see this, notice that for $x\in G/N$, $r(x)$ is the unique element of $D$ such that $x=N r(x)$. Then, for $g\in G$ the following coset equality holds: $x g=Nr(x) g$. Also $xg= N r(xg)$, thus we infer $Nr(x) g= Nr(xg)$, and therefore $r(x) g r^{-1}(xg) \in N$ and is unique, given $x$ and $g$. 
Hence the map $\alpha : G/N\times G\to N$ is well-defined. It also satisfies the cocycle identity ( item (ii) in the Lemma below).  
If $\pi: N\to\mathcal U(H)$ is a unitary representation then $\textup{Ind}_N^G \pi : G\to \mathcal{U}(l^2(G/N, H))$ is a representation of $G$, where the Hilbert space of the induced representation is  $$l^2(G/N, H)=\{  \psi: D\to H \text{ }|\text{ } \sum_{g\in D} \|\psi(g)  \|^2 < \infty  \}$$ with inner product $\ip{\phi}{\psi} = \sum_{g\in D} \ip{ \phi(g)}{\psi(g) }_H $.
\end{example}
We will need the following properties of the cocycle $\alpha$ defined in the example above.
\begin{lemma}\label{coc} In the setting of Example \ref{cocy} we have
	\begin{enumerate}
\item $\alpha(N,n)=n$, for all $n\in N$.
\item $\alpha(x, g_1g_2)=\alpha(x, g_1)\alpha(xg_1,g_2)$, for all $x\in G/N$, $g_{1}, g_2\in G$.
\item $\alpha(N, n g)=n \alpha(N, g)$, for all $n\in N$, $g\in G$. 
\end{enumerate}
\end{lemma}
\begin{proof}
(i) follows from the definition of the map $r$ above and $r(N)=e$.

(ii) we have $\alpha(x, g_1g_2) = r(x) g_1g_2 r^{-1}(x g_1g_2)$.  Next, substitute $r^{-1}(x g_1g_2)$ in the righthand side with $g_2^{-1}r^{-1}(xg_1)\alpha(xg_1, g_2)$ to continue with $$\alpha(x, g_1g_2)=r(x)g_1r^{-1}(xg_1)\alpha(xg_1, g_2)=\alpha(x,g_1)\alpha(xg_1, g_2).$$ 

(iii) follows from the cocycle identity, then (i) and $Nn=N$ for $n\in N$. 
\end{proof}

\begin{theorem}\label{ind} Let $N$ be a subgroup of a countable group $G$ such that  $N$ has finite index in $G$. If $N$ is framenable then $G$ is framenable. 
\end{theorem}
\begin{proof} Notice that $N$ must be countable. Let $D=\{e, g_1, \dots , g_k\}$ be a set of right coset representatives, and assume $N$ is framenable. Let $\pi:N\to\mathcal U(H)$ be a unitary representation with $1_N\prec \pi$,  a countable set $S\subset N$, and a vector $w\in H$ such that $\{\pi_n(w)\}_{n\in S}$ is a weak frame in $H$. We consider the induced representation $\rho:= \textup{Ind}_N^G(\pi)$ corresponding to the cocycle $\alpha$ from Example \ref{cocy}. By the properties of induced representations, because $N$ has finite index in $G$ (see Theorem E.3.1 in \cite{BHV}), from $1_N\prec \pi$ we obtain $1_G\prec\rho$. Thus, it remains to find a weak frame vector $\phi\in l^2(G/N, H)$ with weak frame $\{ \rho_g(\phi) \}_{g\in S'}$, for a countable subset $S'\subset G$. Define $\phi: G/N \to H$ by $\phi(N)=w$, and $\phi(x)=0$, for all $x\neq N$. Notice $\|\phi\|_{l^2}=\|w\|_{H} $ with respect to corresponding Hilbert space norms. Let $P:=\sqcup_{i=1}^k S^{-1}g_i$. Then $P\subset G$ is countable and is a finite union of disjoint subsets of $G$ because $S^{-1}\subset N$ and $\{Ng_i\}_{i=1}^k$ are mutually disjoint.  For the countable set $S':=P^{-1}\subset G$ we verify that $\{\rho_g(\phi)\}_{g\in P^{-1}}$ is a weak frame. We have for all $0\neq \varphi\in l^2(G/N, H)$ :
\begin{align}\label{weakf}
&\sum_{g\in P^{-1}} |\ip{\rho_g\phi }{ \varphi}_{l^2}|^2  = \sum_{g\in P^{-1}} \left|  \sum_{x} \ip{ \phi(x) }{\rho_{g^{-1}} \varphi(x)}_{H} \right|^2  = (\text{only }\phi(N)\neq 0 )     \\
&\sum_{g\in P}|\ip{w}{  \rho_g\varphi(N) }_{H} |^2    =   \sum_{i=1}^k \sum_{g\in S^{-1} g_i} |  \ip{ w }{ \pi[\alpha(N,g ) ] (\varphi (Ng))  }_{H} |^2=    \nonumber \\
&\sum_{i=1}^k \sum_{n\in S}  |\ip{ w }{ \pi[\alpha(N, n^{-1}g_i ) ] (\varphi (Nn^{-1}g_i))  }_{H} |^2 = (\text{by (iii) in Lemma \ref{coc} and } Nn^{-1}=N) \nonumber \\
&\sum_{i=1}^k \sum_{n\in S}  |\ip{ w }{ \pi[n^{-1}\alpha(N,g_i) ] (\varphi (Ng_i))  }_{H} |^2=\sum_{i=1}^k  \sum_{n\in S} | \ip{ \pi_n w }{ \pi[\alpha(N,g_i) ] (\varphi (Ng_i))  }_{H} |^2  \nonumber
 \end{align}
 Rewriting $v_i := \pi[\alpha(N,g_i) ] (\varphi (Ng_i)) \in H$ we see that the latter double sum satisfies
 \[
0<\sum_{i=1}^k  \sum_{n\in S} | \ip{ \pi_n w }{ v_i }_H |^2 <\infty
 \]
 Indeed, we add $k$ sums over $S$, each of which must be finite because $\{\pi_nw\}_{n\in S}$ is a weak frame thus $\sum_{n\in S} | \ip{ \pi_n w }{ v_i } |^2 <\infty$. 
The first inequality is satisfied, else the weak frame property would imply all $v_i=0$, and because $\pi[\alpha(N,g_i) ] $ is unitary we would get all $\varphi(Ng_i)=0$, thus $\varphi=0$. 
 In conclusion $\{\rho_g(\phi)\}_{g\in P^{-1}}$ is a weak frame, and it follows that $G$ is framenable. 
\end{proof}
It would be interesting to show the converse of the theorem above. Then non framenability would  share a similar result with property (T): if $N$ has finite index in $G$ then $N$ has property (T) if and only if $G$ has property (T) (see e.g.\cite{BHV}). Nevertheless, the theorem allows us to extend the framenable class considerably.  
\begin{corollary}\label{lat} If $\Lambda \subset SL(2,\mathbb{R})$ is a countable lattice (i.e., finite co-volume subgroup) then $\Lambda$ is framenable.
\end{corollary}
\begin{proof} Any such $\Lambda$ is a Fuchsian group (e.g. see \cite{Kat}) and thus it admits a finite index subgroup $N$ with a surjective morphism $\varphi : N\to\mathbb{Z}$. By Proposition \ref{surj}, $N$ is framenable, and by Theorem \ref{ind}, the lattice $\Lambda$ is as well.  
\end{proof}

The next result is interesting because it shows that one can extend frame-like vectors from a subgroup representation to a representation of the parent group. 
\begin{theorem}\label{framext} Let $N$ be countable subgroup of a countable group $G$ and $\alpha: G/N\times G\to N$ be any cocycle such that the restriction $\alpha(N, \cdot)_{ \text{  } | \text {  }N}: N\to N$ is bijective. 
If $\pi: N\to\mathcal U(H)$ has a Parseval frame/frame/weak frame vector then so does $ \textup{Ind}_N^G \pi : G \to \mathcal{U}(l^2(G/N, H))$.  
\end{theorem}
\begin{proof} We need not assume $\alpha$ is the cocycle from Example \ref{cocy}, thus Lemma \ref{coc} is not available entirely: (i) is not available but we only need its weaker version, namely that $\alpha(N,\cdot): N\to N$ is bijective. (ii) is unchanged as we do consider a cocycle. Also, (iii) becomes
 $\alpha(N, ng_i)=\alpha(N, n)\alpha(N, g_i)$, for all $n\in N$. 
 Let $D=\{g_k\text{ }|\text{ } k\in\mathbb{N} \}$ be a set of right coset representatives, and $w\in H$ a fixed vector. Then for $\phi\in l^2(G/N, H)$ as in the proof of Theorem \ref{ind}, we have $\|w\|^2= \|\phi\|^2$. 
The calculations from \eqref{weakf}  go through similarly with $S=S^{-1}=N$ and $P=P^{-1}=G$. Because $\alpha(N, \cdot): N\to N$ is bijective, $\alpha(N, n)$ permutes the summation over $n\in N$ and hence the sums do not change. Keeping all calculation and notation similar as in the proof above, in the end we obtain: 
\[
 \sum_{g\in G} |\ip{\rho_g\phi }{ \varphi}_{l^2}|^2 =\sum_{i=1}^{\infty}\sum_{n\in N}|\ip{w}{\pi(\alpha(N,n^{-1}))v_i}_H|^2 =\sum_{i=1}^{\infty}  \sum_{n\in N} | \ip{ \pi_n w }{ v_i }_H |^2
\]
where $v_i$ are as in the proof above, hence 
$\|\varphi\|^2=\sum_{i=1}^{\infty} \| v_i\|^2$.
Now, the last identity shows that if $w$ is a Parseval frame vector or frame vector with bounds $A$ and $B$ or a weak frame vector in $H$ with respect to $\pi$ then $\phi$ is precisely the same in $l^2(G/N, H)$ with respect to $\textup{Ind}_N^G\pi$. 
\end{proof}

\section{Examples}\label{ex}

\begin{example}
The free group $\mathbb{F}_n$, $n\geq 1$ has infinite abelianization, hence by Corollary \ref{abl} (ii)  it is framenable. 
\end{example}

\begin{example}
$SL_2(\mathbb{Z})$  is framenable  being a countable lattice of $SL_2(\mathbb{R})$ hence Corollary \ref{lat} applies. For $n\geq 3$, $SL_n(\mathbb{Z})$ has property (T) (see e.g. \cite{BHV}), hence by Proposition \ref{nonT} it cannot be framenable. The group $SL_2(\mathbb{Z})\ltimes \mathbb{Z}^2$ is framenable by Proposition \ref{products}. Let us point out that the pair $(\mathbb{Z}^2 , SL_2(\mathbb{Z})\ltimes \mathbb{Z}^2) $ satisfies relative property (T) (see \cite{Sh}).   
\end{example}

The reader might have noticed that in item (iii) of Proposition \ref{products} we needed at least one of the groups be countable. The case where both groups are finite could not have been dealt with because the (surjective) morphism considered in that proof has now a finite image. Thanks to Theorem \ref{ind} we can now cover this case as well.
\begin{corollary}\label{ff} If $G$ and $N$ are two finite groups then the free (nonabelian) product group $G\star N$ is framenable. 
\end{corollary}
\begin{proof} The map $\phi: G\star N \to G\times N$, $\phi(g_1n_1 g_2 n_2\dots g_kn_k) = (g_1g_2\dots g_k, n_1n_2\dots n_k)$ is a surjective group morphism, thus $\textup{Ker}\phi$ is a finite index subgroup. By the Kurosh Subgroup Theorem (see e.g. \cite{LS}), $\textup{Ker}\phi$ is  of the form $F(X)\star (\star_i a_i G_1  a_i^{-1})\star(\star_j  b_jN_1b_j^{-1})$, where $F(X)$ is a free group with  generators the set $X$, $G_1 \leq  G$, $ N_1 \leq N$, and $a_i$, $b_i$ in $G\star N$. We argue that both subgroups $G_1$ and $N_1$ are trivial therefore $\textup{Ker}\phi=F(X)$ is a free group, hence framenable (then Theorem \ref{ind} implies that $G\star N $ is framenable). The argument is the same, so we show $N_1=\{e\}$. Because $b_j N_1 b_j^{-1}\leq \textup{Ker} \phi$ we have 
$\phi(b_j N_1 b_j^{-1})= \{(e,e)\}\subset G\times N$. We can proceed similarly for $k=1$ as for arbitrary lengths of $b_j=g_1n_1...g_kn_k\in G\star N$ with $g_i\in G$ and $n_i\in N$. Let $k=1$ and $\phi(g_1n_1  N _1 n_1^{-1}g_1^{-1})=\{(e,e)\}  $. Then $( g_1 g_1^{-1}, n )=(e,e) $, $\forall$ $n\in N_1$, which happens only if $N_1=\{e\}$.  
\end{proof}
We do not have a general result to help transfer framenability from a group to a quotient. The next example depends on the fact that we quotient out by a specific cyclic subgroup of order two.  

\begin{example} The group $PSL(2,\mathbb{Z})$ can be defined as the set of functions  $$\{\phi:\mathbb{R}\to\mathbb{R}\text{ } | \text{ } \phi(x)=\frac{ax+b}{cx+d}, a,b,c,d \in\mathbb{Z}, ad-bc=1\}$$ with respect to composition. 
It can also be viewed as $SL_2(\mathbb{Z})/\{ \pm{I} \}$ where $I$ is the identity $2\times 2$ matrix, with respect to (equivalence class of) matrix multiplication. 
We show that $SL_2(\mathbb{Z})/\{ \pm{I} \}$ admits a finite index subgroup which, under a morphism surjects onto $\mathbb{Z}$. The lattice $SL_2(\mathbb{Z})$ admits a subgroup of finite index with a surjective morphism $\varphi: N\to \mathbb{Z}$. Then $\hat{\varphi}: N/\{ \pm{I} \} \to \mathbb{Z}$, $\hat{\varphi}(g\{\pm{I}\} ):= \varphi(g)$, $g\in N$ is a well-defined, surjective morphism. The well-definedness follows from: $\varphi(-I) + \varphi(-I) =\varphi(I)=0$, hence $\varphi(-I)=0$; then for all $g\in N$, $\varphi(-g)=\varphi(-I \cdot g)=\varphi(-I)+\varphi(g)=\varphi(g)$. Because $N$ has finite index, $N/\{ \pm{I} \}$ has finite index, and is also framenable by Proposition \ref{surj}. It follows from Theorem \ref{ind} that $PSL(2,\mathbb{Z})$ is framenable.  Anotther argument can be given, based on the fact that $PSL(2,\mathbb{Z})$ is isomorphic to the free product $\mathbb{Z}_2\star\mathbb{Z}_3$, hence Corollary \ref{ff} applies. 
\end{example}

\begin{example}
For $\mathbb{F}_n$, $n\geq 2$ the free group on $n$ generators we have the following:

\begin{enumerate}
	\item $\textup{Aut}(\mathbb{F}_n)$ is framenable for $n=2$ and $n=3$. In both cases (see \cite[Corollary 1.4]{GL}),  $\textup{Aut}(\mathbb{F}_n)$ has a finite index subgroup which maps onto a free nonabelian group. Because the latter is framenable, using Proposition \ref{surj} and Theorem \ref{ind} we obtain that  $\textup{Aut}(\mathbb{F}_n)$ is framenable when $n=2$ or $n=3$. 
\item For $n\geq 5$ by remarkable results in \cite{KNO} and \cite{KKN},  $\textup{Aut}(\mathbb{F}_n)$ satifies property (T). By Proposition \ref{nonT} these groups are thus not framenable. 
\item The case  $n=4$ is open in terms of property (T) for $\textup{Aut}(\mathbb{F}_n)$. We do not know if it is (non)framenable. 
\end{enumerate}
\end{example}

\begin{example} The Baumslag-Solitar groups $BS_{p,q}=\langle a,b \text{ }|\text{ }ab^pa^{-1}=b^q \rangle$ have infinite abelianization, therefore are framenable by Corollary \ref{abl} (i).
\end{example}

\begin{example} The braid group $B_n$, $n\geq 3$ is generated by $a_1, a_2,\dots ,a_{n-1}$ subject to relations $a_{i+1}a_ia_{i+1}=a_ia_{i+1}a_i$ for $1\leq i\leq n-2$ and $a_ia_j=a_ja_i$ for $1\leq i,j \leq n-1$, $|i-j| \geq 2$. From the presentation of $B_n$ we deduce that its abelianization is the group $\mathbb{Z}$, hence by Corollary \ref{abl}(i) $B_n$ is framenable. We remark from \cite{Cherix} that for all $n\geq 3$, $B_n$ does not contain subgroups with property (T), but it is not known whether $B_n$ satisfies Haagerup's property for $n\geq 4$.  
\end{example}
For the next example we briefly mention the definitions of the three Thompson's groups. These groups are much studied, especially $F$, due to its (non)amenability question. For more properties and details, the reader may consult \cite{BrG}, \cite{BrS}, and \cite{CaFP}. 

\begin{definition} (i) Thompson's group $F$ is the set of all piecewise linear homeomorphisms from $[0,1)$ to $[0,1)$ that are differentiable except at finitely many dyadic rationals and such that on intervals of differentiability the derivatives are powers of $2$. $F$ admits the following finite presentation
$$F\cong  \langle a, b\text{ } |\text{ } [ab^{-1}, a^{-1}ba] = e, [ab^{-1},a^{-2}ba^2 ]= e] \rangle$$
where $[x,y]=xyx^{-1}y^{-1}$. Its commutator$[F,F]$, denoted by $F'$ in Thompson's group literature, is generated by $[x,y]$, $x, y$ in $F$.

(ii) Thompson's group $T$ consists of the same elements as above, possibly admitting at most one (dyadic) point of discontinuity in $[0,1)$ with discontinuity jump equal to $1$. Notice $F=\{g\in T\text{ }| \text{ }g(0)=0\}$. 
Alternatively, $T$ can be viewed as a subgroup of piecewise homeomorphisms of the circle $\mathbb{S}^1$. $T$ is generated by elements $a$, $b$, and $c$, $c^3=e$, with $a,b$ the generators of $F$, to which more relations are added, see \cite{CaFP}.

(iii) Thompson's group $V$ is the set of all piecewise linear bijections, right-continuous on $[0,1)$, which satisfy same differentiability properties as in (i). 
$V$ is generated and presented by generators $a, b, c, d$, $d^2=e$ with $a, b, c$ generators of $T$, plus more relations.  
\end{definition}
The generators $a,b,c,d$ have explicit copies as piecewise linear functions on the unit interval.  
The groups $V$, $T$ and $[F,F]$ are simple (no non trivial normal subgroup). Also $V$ (and $T$ and $F$) satisfy Haagerup's property, see \cite{Far}. The group $V$ shares a property similar to the free nonabelian group, namely it is $C^{*}$ simple, \cite{BouB}, whereas $C^{*}$ simplicity of $T$ is equivalent to the non amenability of $F$, \cite{BouB} and \cite{HaOl}. For more details on $C^{*}$ simple groups we refer the reader to \cite{deH}.     
 
\begin{example} 
The Thompson's group $F$ is framenable because it has infinite abelianization: $F/[F,F]\cong\mathbb{Z}^2$, see \cite{CaFP}. Hence, Corollary \ref{abl}(i) is applicable. 
\end{example}
\begin{remark}
It would be interesting to test diverse representations of $F$ for the possibility of either condition (ii)-(vi) from Theorem \ref{amen}.  A representation with both almost invariant vectors and  (weak) frame vector would imply amenability.  In the nonamenable direction, because $F$ is ICC, the existence of a (weak) frame representation without almost invariant vectors  would imply $F$ is nonamenable, due to Theorem \ref{thai3}.  Let us remark that representations based on the canonical action of $F$ on the unit interval or the real line may not admit weak frames of the form $\{ \pi_g(w) \}_{g\in F}$. For example, the representation (see \cite{HaOl}) $\pi: F\to\mathcal{U}( l^2(\mathcal{X} ))$, $\mathcal{X}:=\mathbb{Z}[1/2]\cap (0,1)$, $\pi_g(\delta_x):=\delta_{g(x)}$ does not admit weak frame vectors. Notice we have opened the interval $[0,1)$ at the left endpoint else $\delta_0$ is a fixed point and Remark \ref{tri} applies.  
Now, if $w\in l^2(\mathcal{X})$ with $w\neq 0$ then there exists $x_0\in\mathcal{X}$ such that $| \ip{ w }{v}  | > 0$ where $v= \delta_{x_0}$. From the definition of $F$ above we see that we can pick infinitely many $g\in F$ such that $g(x_0)=x_0$ i.e., $\pi_g(v)=v$. This entails 
$$\infty=\sum_{g\in F, g(x_0)=x_0}|\ip{w}{\pi_g(v)}|^2=\sum_{g\in F, g(x_0)=x_0}|\ip{\pi_{g^{-1}} w}{ v}|^2\leq \sum_{g\in F}|\ip{\pi_g w}{v}|^2$$
\end{remark}
\begin{remark}
The tools we have developed in the previous section seem to not capture framenability in the case of simple groups. In such cases, any non trivial surjective morphism must be an isomorphism. This situation might raise the question whether countable, simple, framenable groups exist. Remarkably, there are uncountably many (non isomorphic) infinite, simple, finitely generated, amenable groups, see \cite{JuMo}. 
We end with the following list of problems and questions that we find interesting in order to further develop the connections between group and frame theory. 
\begin{enumerate}
\item Find characterizations of framenability (e.g., in terms of group theoretic properties).
\item  Find a non property (T), non framenable group. 
\item Does there exists a countable Haagerup's property and non framenable group? If yes, then this solves ii) too. 
\item  Are the Thompson's groups $T$ or $V$ framenable? If not, then this solves iii) too.  
Note that $T\cong PPSL(2,\mathbb{Z})$ i.e., piecewise $\frac{ax+b}{cx+d}$ homeomorphisms of $\mathbb{R}\cup\{\infty\}$ with finitely   discontinuity points of the derivative, all inside $\mathbb{Q}\cup\{\infty\}$. This isomorphism is due to Thurston, for more details see for example \cite{HaOl} where other interesting amenability related questions are discussed. 
\item  Our tools do not capture the commutator $[F,F]$ as framenable. We think this should be the case. Recall that the (non)amenability question of $F$ and $[F,F]$ are equivalent, hence we expect $[F,F]$ to be in the larger framenable class. The group $F$ is locally indicable, owing to \cite{ GhS} and Thurston's Stability Theorem, hence the commutator is too (but not  finitely generated).  
\item  Is a locally indicable, infinitely generated countable group framenable? More generally, we can ask the same question of a countable group having the property that any of its finitely generated subgroups is framenable. Notice a group with the latter property cannot satisfy property (T) (it would imply finite generation, and we have seen framenable is opposed to property (T) ). 
\end{enumerate}
\end{remark}


\begin{acknowledgements} We would like to thank the anonymous referees for their valuable suggestions, which helped improve our paper. We also thank Professor Paul Jolissaint who, after this paper was accepted for publication, noted that the conclusion of Theorem \ref{ind} holds more generally when $N$ is co-amenable in $G$, rather than  merely of finite index.  
\end{acknowledgements}

\bibliographystyle{alpha}	
\bibliography{eframes}

\end{document}